\documentclass[12pt]{article}
\usepackage[letterpaper, margin=1in]{geometry}
\usepackage{amsfonts,amsmath,amsthm,amssymb,times}
\usepackage{graphicx}
\usepackage{setspace}
\usepackage{upgreek}
\usepackage{enumerate}
\usepackage{color}
\usepackage[dvipsnames,svgnames,table]{xcolor}

\usepackage{rotating}
\usepackage{subfig}
\usepackage{verbatim}
\usepackage{hyperref}
\usepackage{float}
\usepackage[utf8x]{inputenc}

\hypersetup{colorlinks,linkcolor={Magenta},citecolor={Blue},urlcolor={purple}}
\newtheorem{theorem}{THEOREM}[section]

\newtheorem{definition}{DEFINITION}[section]
\newtheorem{corollary}{COROLLARY}[section]

\newtheorem{remark}{REMARK}[section]

\newtheorem{proposition}{PROPOSITION}[section]

\numberwithin{equation}{section}
\theoremstyle{plain}

\newtheorem{example}{EXAMPLE}[section]
\begin{document}

\title{Anti-MANOVA on Compact Manifolds with Applications to 3D Projective Shape Analysis}

\author{{Hwiyoung Lee, Vic Patrangenaru}
\\ \it{Department of Statistics, Florida State University}}

\maketitle

\begin{abstract}
Methods of hypotheses testing for equality of extrinsic antimeans on compact manifolds are unveiled in this paper. The two and multiple sample problem for antimeans on compact manifolds is addressed for large samples via asymptotic distributions, as well as for small samples using nonparametric bootstrap. An example of face differentiation using 3D VW antimean projective shape analysis for data extracted from digital camera images is also given.\\

\end{abstract}

\section{Introduction}
{As numbers are slowly moving to the back burner of Statistics research, they leave the forefront to objects, represented for simplicity as points on a complete metric space. Such metric spaces have a richer structure than the good old numerical spaces, nevertheless they often see themselves embedded in such spaces or, in the infinite dimensional case, in Hilbert spaces. The new object space structure, leads to new statistics, some of them unimaginable in the flat world of classical nonparametric data analysis. Data sitting on object spaces often time mirrors their topological structure, leading to new location and spread parameters, such as {\em means of finite indices} (see Patrangenaru et al.(2019)\cite{PaBuPaOs:2019}) or {\em anti-covariance matrices} (see Wang and Patrangenaru (2018)\cite{WaPa:2018}), that are giving a fuller description of object data. These new local, and global statistics are delivering a more elegant quantitative and qualitative edge over classical statistics, when it comes to Big Data analysis of Complex type. Along these lines, in this paper we introduce extrinsic anti-MANOVA on object spaces that have a smooth structure.}\\
From its inception, inference on manifolds advanced in two separate directions: one on density estimation (see Kim(1998)\cite{Ki:1998}), the other in parameter estimation (see Hendriks and Landsman(1998)\cite{HeLa:1998}), the reason being that the space of functions defined on a manifold, is a linear space, while the manifold in itself is not, thus calling for a different type of data analysis. Our paper brings this aspect of nonlinearity in {\em extrinsic data analysis}, and calls for more emphasis on the analysis of {\em 3D scenes from digital camera images}, where today is the bulk of data due to technological advances.

{In Section \ref{s:anti}, after introducing the extrinsic antimean, a recently introduced population parameter for a random object on a compact object space (see eg Patrangenaru et al(2016)\cite{PaYaGu:2016}), one derives a CLT for sample extrinsic antimeans. A statistic for two sample tests for extrinsic antimeans on compact manifolds is given in section \ref{ssc:2vw-anti}.
Section \ref{sc4} is dedicated to anti-MANOVA on manifolds. The nonparametric anti-MANOVA test developed in this section, is detailed further in the case of 3D projective shape data in subsection \ref{ssc:manova-3Dproj}, where the 3D projective shape space of projective shapes of landmark configurations of $k$-ads, containing a projective frame at given landmark indices, leads to a anti-MANOVA projective shape analysis on the manifold $(\mathbb RP^3)^{k-5}.$ The asymptotic distributions of a Hotelling like statistic associated with the anti-MANOVA hypotheses testing problem is shown to be key to the data analysis in Theorem \ref{t:chi-antimean}, a result that is further specialized to VW anti-MANOVA in Theorem \ref{t:chi-square-vw}. In Section \ref{sc:faces} we consider an example of 3D projective shape analysis of faces from digital camera imaging data.
The paper concludes with a discussion on future directions in extrinsic object data analysis.}

\section{Extrinsic Antimeans on Object Spaces}\label{s:anti}
{
Fr\'echet (1948)\cite{Fr:1948} noticed that for high complexity data, such as the shape of a random contour, numbers or vectors do not provide a meaningful representation. To investigate these kind of data, he introduced the notion of {\em element}, nowadays called {\em object}. In that paper, he mentioned that an object can represent for example ``the shape of an egg randomly taken from a basket of eggs". Fr\'echet's visionary concepts, were nevertheless hard to handle computationally during his lifetime. It took many decades, until such data became the bread and butter of modern data analysis. In particular various types of shapes of configurations extracted from digital images were represented as points on projective shape spaces (see \cite{MaPa:2005}, \cite{PaLiSu:2010}), on affine shape spaces(see \cite{PaMa:2003}, \cite{Su:2006}), or on Kendall shape spaces (see \cite{Kend:1984}, \cite{DrMa:1998}). To analyze the mean and variance of the random object $X$ on a compact object space $(\mathcal M,\rho)$, Fr\'echet defined what we call now the {\em (second order) Fr\'echet function} given by
\begin{equation}\label{eq:frechet-f1}
\mathcal F(p) =  \mathbb E( \rho^2 (p,x)),
\end{equation}
the maximizers of $\mathcal F$ in \eqref{eq:frechet-f1} forming the {\em Fr\'echet antimean set}. In case $\mathcal M,$ is a smooth manifold, and $\rho = \rho_g$ is the geodesic distance associated with a Riemannian structure $g$ on $\mathcal M,$ {\bf there are no necessary and sufficient conditions for the existence of a unique maximizer} of $\mathcal F$, therefore in general, with the possible exception of compact flat Riemannian manifolds, like high dimensional flat tori, or flat Klein bottles, it is preferred to work with a ``chord" distance on $\mathcal M$ induced  by the Euclidean distance in $\mathbb{R}^{N}$ via an embedding $j: \mathcal{M} \rightarrow \mathbb{R}^{N},$ and in this case, the Fr\'echet function becomes
\begin{equation}\label{eq:frechet-f2}
\mathcal{F}(p)= \int_{\mathcal{M}}  \| j(x) - j(p) \|^{2}_{0} Q(dx),
\end{equation}
where $\| \cdot\|_0$ is the Euclidean norm in  $\mathbb{R}^{N}$, $Q = P_X$ is the probability measure on $\mathcal M,$ associated with the r.o. $X$ on $\mathcal M.$ In this setting, if the extrinsic antimean set has one point only, that point is called {\bf extrinsic antimean} of $X,$ and is labeled $\alpha\mu_{j,E}(Q)$, or simply $\alpha\mu_E,$ when $j$ and $Q$ are known. This happens iff the mean vector $\mu$ of $j(X)$ is {\em $\alpha j$-nonfocal}, meaning that its farthest projection on $j(\mathcal M)$, $P_{F,j}(\mu),$ is well defined (see Patrangenaru et al.(2016)\cite{PaYaGu:2016}).
Also, given $X_1, \dots, X_n$ i.i.d.r.v.'s from $Q$, their {\em extrinsic sample antimean (set)} is the extrinsic antimean (set) of the empirical distribution $\hat Q_n = {\frac{1}{n}} \sum_{i=1}^n \delta_{X_i}$ (see eg Patrangenaru et al (2016)\cite{PaYaGu:2016}).
\begin{example}\label{eg:VW-antimean}(see Wang and Patrangenaru(2018)\cite{WaPa:2018})
Assume $Q$ is a probability distribution on the complex projective space $\mathbb{C}P^{k-2}$ and $j$ is its VW embedding, given by $j([z])=\frac{1}{z^*z}zz^*.$ Let $\{[Z_{r}], \| Z_{r} \|, r = 1, \dots, n\}$ be i.i.d.r.o.'s from $Q$. We say that $Q$ is $\alpha$ VW-nonfocal if $Q$ is $\alpha j$-nonfocal, w.r.t. the VW embedding $j$.  Then $(a)$ $Q$ is $\alpha$ VW-nonfocal iff $\lambda$, the smallest eigenvalue of $E[Z_{1}Z_{1}^{*}]$ is simple and in this case $\alpha \mu _{j,E}{Q} = [m],$ where $m$ is an eigenvector of $E[Z_{1}Z_{1}^{*}]$ corresponding to $\lambda$, with $\parallel m \parallel = 1$ and $(b)$ The sample VW antimean $\alpha \overline{X}_{E} = {[m]}$, where $m$ is an eigenvector of norm 1 of $J = \frac{1}{n} \sum^n_{i=1}Z_i Z^*_i$, $\| Z_i\| = 1, i = 1, \dots, n$, corresponding to the smallest eigenvalue of $J.$
\end{example}

Using a moving frame approach $\grave{a}$ la analysis of extrinsic means in Bhattacharya and Patrangenaru (2003,2005)\cite{BhPa:2003,BhPa:2005}, one may develop a methodology for extrinsic antimean estimation (see Patrangenaru et al(2016a)\cite{PaGuYa:2016}). Assume $j$ is an embedding of a $d$-dimensional compact manifold $\mathcal M$ in $\mathbb R^N$, and $Q$ is a $\alpha j$-nonfocal probability measure on $\mathcal M$ such that $j(Q)$ has finite moments of order 2. Let $\mu$ and $\Sigma$ be the mean and covariance matrix of $j(Q)$ regarded as a probability measure on $\mathbb R^N$. Let $\alpha\mathcal F$ be the set of $\alpha j$-focal points of $j(M)$, and let $P_{F,j}:\alpha\mathcal F^c \rightarrow j(M)$ be the farthest projection on $j(M)$.
Assume $x \rightarrow (f_1(x),\ldots,f_d(x))$ is a local frame field on an open subset of $M$ such that for each $x\in M$, $(d_xj(f_1(x)),\ldots,d_xj(f_d(x)))$ are orthonormal vector in $\mathbb R^N$. A local frame field $p  \rightarrow (e_1(p),e_2(p),\ldots,e_N(p))$ defined on an open neighborhood $U\subseteq \mathbb R^N$ is {\em adapted to the embedding} $j$ if it is an orthonormal frame field and $\forall x\in j^{-1}(U),e_r(j(x))=d_xj(f_r(x)),r=1,\ldots,d$.

Let $e_1,e_2,\ldots,e_N$ be the canonical basis of $\mathbb R^k$ and assume $(e_1(p),e_2(p),\ldots, e_N(p))$ is an adapted frame field around $P_{F,j}(\mu)=j(\mu_{\alpha E}).$ Then $d_\mu P_{F,j}(e_b)\in T_{P_{F,j}(\mu)}j(M)$ is a linear combination of $e_1(P_{F,j}(\mu)),e_2(P_{F,j}(\mu)),\ldots, e_d(P_{F,j}(\mu))$:
\begin{equation}
d_\mu P_{F,j}(e_b)=\sum_{a=1}^d(d_\mu  P_{F,j}(e_b))\cdot e_a( P_{F,j}(\mu))e_a( P_{F,j}(\mu)).
\end{equation}
By the delta method, $n^{1/2}(P_{F,j}(\overline{j(X)})-P_{F,j}(\mu))$ converges weakly to $N_N(0_N,\alpha\Sigma_\mu),$ where $\overline{j(X)}=\frac{1}{n}\sum_{i=1}^nj(X_i)$ and
\begin{equation}
\begin{split}
\alpha\Sigma_\mu=[\sum_{a=1}^d d_\mu P_{F,j}(e_b)\cdot e_a(P_{F,j}(\mu))e_a(P_{F,j}(\mu))]_{b=1,\ldots,N}\\
 \times\Sigma[\sum_{a=1}^d d_\mu P_{F,j}(e_b)\cdot e_a(P_{F,j}(\mu))e_a(P_{F,j}(\mu))]^T_{b=1,\ldots,N}
\end{split}
\end{equation}
Here $\Sigma$ is the covariance matrix of $j(X_1)$ w.r.t the canonical basis $e_1,e_2,\ldots, e_N$. The asymptotic distribution $N_N(0_N,\alpha\Sigma_\mu)$ is degenerate and the support of this distribution is on $T_{P_{F,j}}j(M)$, since the range of $d_\mu P_{F,j}$ is a subspace of $T_{P_{F,j}(\mu)}j(M)$. Note that $d_\mu P_{F,j}(e_b)\cdot e_a(P_{F,j}(\mu))=0$ for $a = d+1, \ldots, N.$ The tangential component $tan(v)$ of $v\in \mathbb R^N$, w.r.t. the basis $e_a(P_{F,j}(\mu))\in T_{P_{F,j}(\mu)}j(M),a=1,\ldots,d$ is given by

\begin{equation}\label{eq:tan}
tan(v)=[e_1(P_{F,j}(\mu))^Tv,\ldots,e_d(P_{F,j}(\mu))^Tv]^T.
\end{equation}

From \eqref{eq:tan}, $(d_{\alpha\mu_E}j)^{-1}(tan(P_{F,j}(\overline{(j(X))})-P_{F,j}(\mu)))=\sum_{a=1}^d \overline X_j^a f_a$ has the following covariance matrix w.r.t. the basis $f_1(\alpha\mu_E),\ldots,f_d(\alpha\mu_E):$
\begin{equation}\label{eq:anti-cov}
\begin{split}
\alpha\Sigma_{j,E}=e_a(P_{F,j}(\mu))^T\alpha\Sigma_\mu e_b(P_{F,j}(\mu))_{1\leq a,b\leq d}\\
=[\Sigma d_\mu P_{F,j}(e_b)\cdot e_a(P_{F,j}(\mu))]_{a=1,\ldots,d}\Sigma\\
\times [\Sigma d_\mu P_{F,j}(e_b)\cdot e_a(P_{F,j}(\mu))]^T_{a=1,\ldots,d}
\end{split}
\end{equation}
}
{
\begin{definition}

The matrix $\alpha\Sigma_{j,E}$ given above is the {\em extrinsic anticovariance matrix} of the $\alpha j$ -nonfocal distribution $Q$ (of $X_1$) w.r.t. the basis $f_1(\mu_{\alpha E}),\ldots,f_d(\mu_{\alpha E}).$ When $j$ is fixed, the subscript $j$ in $\alpha\Sigma_{j,E}$ will be omitted. If rank $\alpha\Sigma_\mu=d$, then $\alpha\Sigma_{j,E}$ is invertible and we define the $j$-standardized sample antimean vector
\begin{equation}
a\overline{Z}_{j,n}=:n^{1/2}\alpha\Sigma_{j,E}^{-1/2}(\overline X_j^1,\ldots,\overline X_j^d)^T.
\end{equation}
\end{definition}

We recall the following 

\begin{theorem}\label{t:CLT-antimean}(Patrangenaru et al(2016)\cite{PaGuYa:2016})
Assume $\{X_r\}_{r=1,\ldots,n}$ is a random sample from the $\alpha j$-nonfocal distribution $Q$. Let $\mu=E(j(X_1))$ and let $(e_1(p),e_2(p),\ldots,e_k(p))$ be an orthonormal frame field adapted to $j$. Then (a) the tangential component of the extrinsic sample antimean $a\overline X_{E}$ has asymptotically a normal distribution in the tangent space to the $d$ dimensional manifold $M$ at $\alpha\mu_{E}(Q)$ with mean $0_d$ and covariance matrix $n^{-1}\alpha\Sigma_{j,E}$, and (b) if $\alpha\Sigma_{j,E}$ is nonsingular, the $j$-standardized mean vector $\overline{\alpha Z}_{j,n}$ converges weakly to a random vector with a multivariate $N_d(0_d,I_d)$ distribution.
\end{theorem}

As a particular case, when $j$ is the inclusion map of a submanifold $\mathcal M$ of $\mathbb R^k$, we get the following result for $\alpha$-nonfocal distributions on $\mathcal M:$

\begin{corollary}
Assume $M\subseteq \mathbb R^k$ is a $d$-dimensional closed submanifold of $\mathbb R^k.$ Let $\{X_r\}_{r=1,\ldots,n}$ be i.i.d.r.o's from the nonfocal distribution $Q$ on $M$, and let $\mu=E(X_1)$ and assume the covariance matrix $\Sigma$ of $j(Q)$ is finite. Let $(e_1(p),e_2(p),\ldots,e_N(p))$ be an orthonormal frame field adapted to $M$. Let $\alpha\Sigma_E:=\alpha\Sigma_{j,E}$, where $j: M\rightarrow \mathbb R^N$ is the inclusion map. Then (a) $n^{1/2}tan(j(a\overline X_E)-j(\alpha\mu_E))$ converges weakly to $N_d(0_d,\alpha\Sigma_E)$, and (b) if $\alpha\Sigma_\mu$ induces a nonsingular bilinear form on $T_{j(\mu_{\alpha E})}j(M)$, then $\|\overline {AZ}_{j,n}\|^2$ converges weakly to the chi-square distribution $\chi_d^2.$
\end{corollary}

The CLT for extrinsic sample antimeans can not be used to construct
confidence regions for extrinsic antimeans, since the population extrinsic covariance matrix is a nuisance parameter.
We then consider a consistent
estimator of $\alpha\Sigma_{j,E}$ as follows. Note that $\overline {j(X)}$  is a
consistent estimator of $\mu$, $d_{\overline{j(X)}}P_{F,j}\to_P d_\mu P_{F,j}$, and
$e_a(P_{F,j}(\overline{j(X)}))\to_P e_a(P_{F,j}(\mu))$ and
\begin{equation}
S_{j,n} = n^{-1}
\sum (j(X_r) - \overline{j(X)})(j(X_r) - \overline{j(X)})^T
\end{equation}
 is a consistent
estimator of $\alpha\Sigma_\mu$. It follows that
\begin{eqnarray}
\left[ \sum_{a=1}^d d_{\overline{j(X)}} P_{F,j} (e_b) \cdot
e_a(P_{F,j}(\overline{j(X)}))e_a(P_{F,j}(\overline{j(X)}))\right]S_{j,n} \nonumber \\
\left[ \sum_{a=1}^d d_{\overline{j(X)}} P_{F,j} (e_b) \cdot
e_a(P_{F,j}(\overline{j(X)}))e_a(P_{F,j}(\overline{j(X)}))\right]^T\nonumber
\end{eqnarray}
is a consistent estimator of  $\alpha\Sigma_{\mu}$, and $tan_{P_{F,j}(\overline{j(X)})}v$
is a consistent estimator of $tan(v)$.

Therefore if we take the components of the bilinear form associated
with this matrix  w.r.t.\\
$e_1(P_{F,j}(\overline{j(X)})),e_2(P_{F,j}(\overline{j(X)})),...,e_d(P_{F,j}(\overline{j(X)}))$,
we get a consistent estimator of $\alpha\Sigma_{j,E}$
\begin{eqnarray} \label{eq:ext-covariance}
{aS_{E,n}}=\nonumber \\
= [[ \sum d_{\overline{j(X)}} P_{F,j}
(e_b) \cdot e_a(P_{F,j}(\overline{j(X)}))]_{a=1,...,d}]
\cdot\nonumber \\
\cdot S_{j,n} [[\sum d_{\overline{j(X)}} P_{F,j}
(e_b) \cdot e_a(P_{F,j}(\overline{j(X)}))]_{a=1,...,d}]^T.
\end{eqnarray}
\begin{remark} As a result, if we assume that $j: M \to \mathbb{R}^N$ is an embedding
of $\mathcal M$ in $\mathbb{R}^k$ and $\{X_r\}_{r=1,...,n}$ is are i.i.d.r.o.'s
from the $\alpha j$-nonfocal distribution $Q$, and $\mu =E(j(X_1))$,
$j(X_1)$ has finite second order moments, and $\alpha\Sigma_{j,E}$ of $X_1$ is nonsingular, then if
$(e_1(p),e_2(p) ,....,e_k(p))$ be an orthonormal frame field adapted
to  $j$, it follows that ${aS_{E,n}}$ is \eqref{eq:ext-covariance}, then for $n$ large enough, ${aS_{E,n}}$ is nonsingular with
probability converging to one,  and (a)
\begin{equation}
n^{1\over 2} {aS_{E,n}}^{-{1\over 2}}(P_{F,j}(\overline{j(X)}) -P_{F,j}(\mu))
\end{equation}
 converges weakly to a $N ( 0_d ,I_d)$ distributed r.vector,  so that
\begin{equation}
n \|{aS_{E,n}}^{-{1\over 2}} tan(P_{F,j}(\overline{j(X)}) - P_{F,j}(\mu))\|^2
\end{equation}
converges weakly to a $\chi^2_d$ distributed r.v., and (b) the statistic
\begin{equation}
n^{1\over 2} {aS_{E,n}}^{-{1\over 2}} tan_{P_{F,j}(\overline{j(X)})}
(P_{F,j}(\overline{j(X)}) -P_{F,j}(\mu))
\end{equation}
 converges weakly to a $N ( 0_d ,I_d)$ r. vector,  so that
\begin{equation}
n \|{aS_{E,n}}^{-{1\over 2}} tan_{P_{F,j}(\overline{j(X)})}
(P_{F,j}(\overline{j(X)}) -P_{F,j}(\mu))\|^2
\end{equation}
converges weakly to  $\chi^2_d$ distributed r.v.
\end{remark}
\begin{corollary} Under the hypothesis above, a confidence
region for $\alpha\mu_E$ at asymptotic level
$1 - \alpha$ is given by (a) $C_{n,\alpha}:=j^{-1}(U_{n,\alpha}),$ where
$$U_{n,\alpha}=\{\mu \in j(M): n \|{aS_{E,n}}^{-{1\over 2}} tan(P_{F,j}(\overline{j(X)}) -P_{F,j}(\mu)\|^2 \leq \chi_{d,1-\alpha}^2\},$$ or by
(b) $D_{n,\alpha}:=j^{-1}(V_{n,\alpha}),$ where
$$V_{n,\alpha}=\{\mu \in j(M): n \|{aS_{E,n}}^{-{1\over 2}} tan_{P_{F,j}(\overline{j(X)})}
(P_{F,j}(\overline{j(X)}) -P_{F,j}(\mu)\|^2 \leq \chi_{d,1-\alpha}^2\}.$$
\end{corollary}

\par At this point we recall the steps that one takes to obtain
a bootstrapped statistic from a pivotal statistic. If $\{X_r\}_{r=1,...,n}$
is a random sample from  the unknown
distribution $Q$, and $\{X^*_r\}_{r=1,...,n}$ is a random sample
from the empirical $\hat
Q_n$, conditionally given $\{X_r\}_{r=1,..,.n}$, then  the statistic
$$T(X,Q) = n \|{aS_{E,n}}^{-{1\over 2}}
tan(P_{F,j}(\overline{j(X)}) -P_{F,j}(\mu))\|^2$$
given above  has the bootstrap analog
\begin{eqnarray}
T(X^*,\hat{Q}_n) = n\|{aS_{E,n}^*}^{-{1\over 2}} \nonumber \\
tan_{P_{F,j}(\overline{j(X)}))}(P_{F,j}(\overline{j(X^*)})
- P_{F,j}(\overline{j(X)}))\|^2.
\end{eqnarray}

Here ${aS_{E,n}^*}$ is obtained from ${aS_{E,n}}$ substituting
$X^*_1,....., X^*_n$ for $X_1,.....X_n$, and
$T(X^*,\hat{Q}_n)$ is obtained from $T(X,Q)$ by substituting
$X^*_1,.....,  X^*_n$ for $X_1,...., X_n$, $\overline{j(X)})$ for $\mu$
and $aS_{E,n}^*$ for ${aS_{E,n}}$. \\
The same procedure can be used for the vector valued statistic
\begin{equation}
V(X,Q) = n^{1\over 2} {aS_{E,n}}^{-{1\over 2}}
tan(P_{F,j}(\overline{j(X)}) -P_{F,j}(\mu)),
\end{equation}
and as a result we get the bootstrapped statistic
\begin{eqnarray}
V^*(X^*,\hat{Q}_n) = n^{1\over 2}{aS_{E,n}^*}^{-{1\over 2}} \nonumber \\
tan_{P_{F,j}(\overline{j(X^*)}))}(P_{F,j}(\overline{j(X^*)}) - P_{F,j}(\overline{j(X)})).
\end{eqnarray}

We then  obtain the following results:
\begin{theorem}\label{t:boot}
 Let $\{X_r\}_{r=1,...,n}$ be i.i.d.r.o.'s
from the $\alpha j$-nonfocal distribution $Q$,which has a nonzero absolutely
continuous component w.r.t.\ the volume measure on $M$ induced by $j$.
Let $\mu =E(j(X_1))$ and  assume  the extrinsic covariance matrix $\Sigma_{j,E}$ is
nonsingular and let $(e_1(p),e_2(p) ,....,e_k(p))$ be an orthonormal
frame field adapted to  $j$.
Then the distribution function of
\begin{equation}
n\|{aS_{E,n}}^{-{1\over 2}} tan(P_{F,j}(\overline{j(X)}) -P_{F,j}(\mu))\|^2
\end{equation}

can be approximated by the bootstrap distribution function of
\begin{equation}
n\| {aS_{E,n}^*}^{-{1\over2}} tan_{P_{F,j}(\overline{j(X)}))}
(P_{F,j}(\overline{j(X^*)}) - P_{F,j}(\overline{j(X)}))\|^2
\end{equation}

with a coverage error $0_p(n^{-2}).$
\end{theorem}

\par A practical method of finding a nonpivotal confidence region for the extrinsic antimean, consists of considering a chart, defined around all the bootstrap antimeans $P_{F,j}(\overline{j(X^*)}$ that the antimeans of resamples to $\mathbb{R}^d$; such a confidence region in terms of simultaneous confidence intervals.

\subsection{VW antimeans on $\mathbb{R}P^m$}

In this section we consider the case when $\mathcal{M}= \mathbb{R}P^m$ is the real projective space, set of $1$-dimensional linear subspaces of $\mathbb{R}^{m+1}.$ $( \mathbb{R}P^{m},\rho_0) $ is a compact space with $\rho_0$ the chord distance induced by the Veronese Whitney (VW) embedding in the space of $(m+1)\times(m+1)$ positive semi-definite symmetric matrices, $j:\mathbb{R}P^m \to S_+(m+1, \mathbb{R})$ given by
\begin{equation}\label{eq:VW}
j([x])= xx^{T}, \|x\|=1
\end{equation}
We first must recall some properties of the VW embedding. It is an equivariant embedding, this means that it acts on the left on $S_{+}(m+1, \mathbb{R}),$ the set of nonnegative definite symmetric matrices with real coefficients, by
\begin{align}
T \cdot A= T A T^{T},	~~\forall~T\in SO(m+1), \forall~A\in S_+(m+1, \mathbb{R})\notag\\
j(T \cdot [x])= T \cdot j([x]),~~ \forall~[x] \in \mathbb{R}P^{m}, \notag
\end{align}
where $T \cdot [x] = [Tx].$\\
Also $j(\mathbb{R}P^{m})=\{ A \in S_{+}(m+1, \mathbb{R}): rank (A)=1, Tr(A)=1.\}$
And the set $\mathcal{F}$ of $j$-focal points of $j(\mathbb{R}P^{m})$ in $S_{+}(m+1, \mathbb{R}),$ is the set of matrices in $S_{+}(m+1, \mathbb{R})$ whose largest eigenvalues are of multiplicity at least 2. The induce distance id defined as follow; for $A,~B ~\in ~S(m+1, \mathbb{R})$ we define $d_0(A,B)= tr((A-B)^2).$
\bigskip
Recall that if $\mu=E(XX^T)$ is the mean of $j(Q)$ in $\mathbb{R}^{N}.$
\begin{equation}
\mathcal{F}([p]) = \| j([p])- \mu \|_{0}^{2} +\int_{\mathcal M}~\| \mu- j([x]) \|_0^{2} Q(dx)
\end{equation}
And $\mathcal{F}([p])$ is maximized if and only if $ \| j([p])- \mu \|_{0}^{2}$ is maximize with respect to $[p] \in \mathcal{M}.$
\begin{proposition}
\begin{enumerate}[(i)]
\item $(\alpha F)^c,$ set of $\alpha VW$-nonfocal points in $S_{+}(m+1, \mathbb{R}),$ is made of matrices in whose smallest eigenvalue has multiplicity 1.
\item The projection $P_{F,j}: (\alpha F)^c \to j(\mathbb{R}P^{m})$ assigns to each nonnegative definite symmetric matrix $A$, of rank 1, with a smallest eigenvalue of multiplicity 1, the matrix $j([\nu])$, where $\| \nu \|=1$  and $\nu$ is an eigenvector of  $A$ corresponding to that eigenvalue.
\end{enumerate}
\end{proposition}
We now have the following;
\begin{proposition}\label{p:aVW}
Let $Q$ be a distribution on $\mathbb{R}P^{m}.$
\begin{enumerate}
\item The VW-antimean set of a random object $[X], X^TX=1$ on $\mathbb RP^m,$ is the set of points $p = [v]\in V_1,$ where $V_1$ is the eigenspace corresponding to the smallest eigenvalue $\lambda(1)$ of $E(XX^T).$
\item If in addition $Q = P_{[X]}$ is  $\alpha VW$-nonfocal, then
$$\alpha\mu_{j, E}(Q)=j^{-1}(P_{F,j}(\mu))=\gamma(1) $$ where $(\lambda(a), \gamma(a))$, $a=1,.., m+1$ are eigenvalues in increasing order and the corresponding unit eigenvectors of $\mu=E(X X^{T}).$
\item Let $x_1,\dots,x_n$ be random observations from a distribution $Q$ on $\mathbb RP^m,$ such that $\overline{j(X)}$ is $\alpha$ VW-nonfocal. Then the VW sample antimean of $x_1,\dots,x_n$ is given by;
$$a \overline{x}_{j, E} =j^{-1}(P_{F,j}(\overline{j(x)}))= g(1) $$ where $(d(a), g(a))$ are the eigenvalues in increasing order and the corresponding unit eigenvectors of $\displaystyle{J = \sum_{i=1}^{n} x_{i} x^{T}_{i}}.$
\end{enumerate}
\end{proposition}

\section{ Hypothesis testing for two VW antimean projective shapes}\label{ssc:2vw-anti}

The real projective space, $\mathbb{R}P^{m},$ is the building block in the geometric structure of the projective shape space; the projective shape
space of $k$-ads $(x_1, \dots, x_k$ in $\mathbb RP^m$ that include a {\em projective frame} at given fixed indices, can be identified with
$(\mathbb RP^m)^q,$ where $q = k-m-2.$ (see Mardia and Patrangenaru (2005)\cite{MaPa:2005}). This space is embedded via the VW-embedding
$j_q:(\mathbb{R}P^m)^q \to (S_+(m+1, \mathbb{R}))^q$ as follows:
\begin{equation}\label{eq:VW-emb} j_q([x_1], \dots, [x_q]) = (j([x_1]), \dots, j_q([x_q])),
\end{equation}
where $j$ is the VW embedding of  $\mathbb{R}P^{m},$ given in \eqref{eq:VW}.

Assume that for $a=1,2, Q_a$ are $\alpha$VW-nonfocal. We are now interested in the hypothesis testing problem:
 \begin{equation}\label{eq:hypo-means}
H_0 :\  \alpha\mu_{1,E}= \alpha\mu_{2,E} \ \text{vs.} H_a :\ \alpha\mu_{1,E} \ne \alpha\mu_{2,E},
\end{equation}
For $m = 3,$ the hypothesis \eqref{eq:hypo-means} is equivalent to the following
\begin{equation}\label{eq:hypo-means-Lie}
H_0 :\ \alpha\mu_{2,E}^{-1} \odot \alpha\mu_{1,E}= 1_{q}  \ \text{vs.} H_a :\ \alpha\mu_{2,E}^{-1} \odot \alpha\mu_{1,E} \ne 1_{q}
\end{equation}
\begin{enumerate}
\item Let $n_+ = n_1 + n_2$ be the total sample size, and assume $\lim_{n_+ \to \infty} \frac{n_{1}}{n_+} \to \lambda \in (0,1)$. Let $\varphi$ be the log chart defined in a neighborhood of
$ 1_{q}$ (see Helgason (2001)),  with $\varphi(1_{q})=0.$  Then, under $H_0$
\begin{align}
n_+^{1/2}~\varphi(a\bar{Y}^{-1}_{n_2,E} \odot a\bar{Y}_{n_1,E}) \to_{d} \mathcal{N}_{3q}(0_{3q} , \tilde\Sigma_{j_q}),
\end{align}
for some covariance matrix $ \tilde\Sigma_{j_{q}}.$
\item Assume in addition that for $a=1,2$ the support of the distribution of $Y_{a,1}$ and the VW anti mean $\alpha\mu_{a,E}$ are included in the domain of the chart $\varphi$ and $\varphi(Y_{a,1})$ has an absolutely continuous component and finite moment of sufficiently high order. Then the joint distribution\\
\begin{equation}\label{eq:antimean-VW1}
aV = {n+}^{1\over 2} \varphi(a\bar{Y}^{-1}_{n_2,E} \odot a\bar{Y}_{n_1,E})
\end{equation}
can be approximated by the bootstrap joint distribution of \\
\begin{equation}\label{eq:antimean-VW2}
\displaystyle{aV^{*}={n+}^{1/2}~\varphi(a\bar{Y^{*}}^{-1}_{n_2,E} \odot a\bar{Y}^{*}_{n_1,E}) }
\end{equation}
 \end{enumerate}
Now, from proposition \ref{p:aVW}, we get the following result that is used for the computation of the VW sample antimeans:
\begin{proposition}\label{p:aVWprsh}
 follows that given a random sample from a distribution $Q$ on $\mathbb R P^m,$ if $J_s, s = 1,\dots, q$ are the matrices $J_{s}=n^{-1}\sum_{r=1}^{n}X_{r}^{s}(X_{r}^{s})^{T},$ and if for $a = 1,\dots, m+1, d_{s}(a)$ and $g_{s}(a)$ are the eigenvalues in increasing order and  corresponding unit eigenvectors of $J_{s},$ then the VW sample antimean $a\bar{Y}_{n,E}$ is given by
\begin{equation}
\label{eq:VW-sample-mean}
a\bar{Y}_{n,E} = ([g_{1}(1)],\dots,[g_{q}(1)]).
\end{equation}
\end{proposition}

\section{ Extrinsic Anti-MANOVA on Compact Manifolds }\label{sc4}

\noindent Consider an embedding $j: \mathcal{M} \to \mathbb R^N,$ of a compact manifold $\mathcal{M}$ of dimension $d.$ For a=1, \dots, g, let $X_{a,1}, \ldots,X_{a,n_a}$ be i.i.d.r.o.'s on $\mathcal{M}$ with the probability measure $Q_{a}=P_{X_{a,1}}$ being $\alpha$j-nonfocal. Let $\alpha\mu_{a,E}$ be the extrinsic antimean of $Q_a,$ and $a\bar{X}_{a,E}$ be the sample extrinsic antimean of $X_{a,1}, \ldots,X_{a,n_a}.$
Define the {\em pooled extrinsic antimean} with weights $\lambda = (\lambda_1, \dots, \lambda_g)$, denoted by $\alpha\mu_{E}(\lambda)$ given by
\begin{equation}
j(\alpha\mu_{E})= P_{F,j}(\lambda_1 j(\alpha\mu_{1,E}) + \cdots + \lambda_{g}j(\alpha\mu_{g,E})). \label{E_Polled_mean}
\end{equation}
Likewise, the {\em pooled sample extrinsic antimean }, denoted by $a\bar{X}_{E} ~\in~\mathcal M$ is given by
\begin{equation}\label{E_Polled_sample_mean}
j(a\bar{X}_{E})= P_{F,j}(a\overline{j^{(p)}(X)}),
\end{equation}
where $a\overline{j^{(p)}(X)}=\frac{n_1}{n} j(a\bar{X}_{1, E}) + \cdots + \frac{n_g}{n}j(a\bar{X}_{g, E}).$ Here it is assumed that
$a\bar{X}_{a, E}$, the extrinsic sample antimean for the $a$-th sample is well defined, and $n=\sum_{a=1}^{g} n_a;$ The weights are in this case $\hat{\lambda}_a = \frac{n_a}{n}, a = 1, \dots, g.$ Under the null hypothesis

\begin{equation} H_0: \alpha\mu_{1,E}= \cdots = \alpha\mu_{g, E},\label{eq:h0-MANOVA}\end{equation}
and the usual alternative, for $b=1,\dots, g,$ we consider:

$\displaystyle{aS_{b}= (n_b)^{-1} \Sigma_{i=1}^{n_b}(j(X_{b,i})- j(a\bar{X_b}_{E})) (j(X_{b,i})- j(a\bar{X_b}_{E}))^T   }$ as a consistent estimator of $\alpha\Sigma_{b}.$ Also note that $\tan_{j(a\bar{X}_{E}) } \nu$ is a consistent estimator of $\tan_{P_{F,j}(\mu)} \forall \nu \in \mathbb{R}^{N}.$

 It follows that the {\em extrinsic sample anticovariance matrix} $aS_{b,E},$ given by
 {
\begin{align}
aS_{b,E}&=\left[ \left[\sum_{a=1}^{d} d_{a\overline{j^{(p)}(X)}}P_{F,j}(e_a) \cdot e_{i}(j(a\bar{X}_{E}))  ~e_{i}(j(a\bar{X}_{E}))\right]_{i=1,...,d} \right] \cdot ~S_{n_b}\notag\\
&~~~~\left[ \left[\sum_{a=1}^{d} d_{a\overline{j^{(p)}(X)}}P_{F,j}(e_a) \cdot e_{i}(j(a\bar{X}_{E})) e_{i}(j(a\bar{X}_{E})) \right]_{i=1,...,d} \right]^T \notag
\end{align}
is a  consistent estimator for $\alpha\Sigma_{b,E}$
}

\begin{theorem}\label{t:chi-antimean}
Assume $j: \mathcal{M} \rightarrow \mathbb{R}^N$ is an embedding of the compact manifold $\mathcal{M}$. For $a=1,...,g,$ let $\{ X_{a,i} \}_{i=1, \dots, n_a}$ be i.i.d.r.o.'s from the $j$-nonfocal distributions $\mathcal{Q}_a$ on $\mathcal{M}$. Let $\mu_a = E(j(X_{a,1}))$ and assume the extrinsic anticovariance matrices $\alpha\Sigma_{a,E}$ of $X_{a,1}$ are nonsingular. We also let $\left( e_{1}(p), ...., e_{N}(p) \right)$, for $p~\in \mathcal{M}$ be an orthonormal frame field adapted to $j$ defined in an open neighborhood of the pooled extrinsic antimean and of the set of extrinsic population antimeans. Assume that $\frac{n_a}{n}\rightarrow \lambda_a>0,\ as\ n\rightarrow \infty, \forall a = 1, \dots, k.$ Then  under \eqref{eq:h0-MANOVA},
\begin{equation}\label{eq:manova-large}
\sum_{a=1}^g n_a~ \tan_{j(\alpha\mu_E)} (j(a\bar{X}_{a,E})-j(a\bar{X}_E))^T aS_{b,E}^{-1}~\tan_{j(\alpha\mu_E)} (j(a\bar{X}_{a,E})-j(a\bar{X}_E))   \to_d    \chi_{gd}^2.
\end{equation}
and
\begin{equation}\label{eq:manova-large2}
\sum_{b=1}^g n_b~ \tan_{j(a\bar{X}_{E})} (j(a\bar{X}_{b,E})-j(a\bar{X}_E))^T aS_{b,E}^{-1}~\tan_{j(a\bar{X}_{E})} (j(a\bar{X}_{b,E})-j(a\bar{X}_E))   \to_d    \chi_{gd}^2.
\end{equation}

\end{theorem}
\begin{proof}
Recall from Patrangenaru et al (2016)\cite{PaGuYa:2016}, that we have
\begin{align}
\sqrt{n_b}~\tan_{j(\alpha\mu_{E})}(j(a\bar{X}_{b,E})-j(\mu_{E})) \rightarrow_d N(0_d,a\Sigma_{b,E}), ~~~for~ b=1,2,...,g
\end{align}
where
\begin{eqnarray}\label{eq:anti-cov-group}
a\Sigma_{b,E} =\left[ \left[  \sum d_{\mu}P_{F,j}(e_b) \cdot e_{k}(P_{F,j}(\mu)) \right]_{k=1,...,d}\right]\cdot \nonumber\\ \cdot \Sigma_{b}~~\left[\left[ \sum d_{\mu}P_{F,j}(e_b) \cdot e_{k}(P_{F,j}(\mu))  \right]^T_{k=1,...,d} \right].
\end{eqnarray}
 In \eqref{eq:anti-cov-group} $\mu= \lambda_1 j(\mu_{1,E}) + \cdots + \lambda_{g}j(\mu_{g,E})$ and the $\Sigma_{a}$'s are the covariance matrices of the $j(X_{a,1})$'s with respect to the canonical basis $e_1,...,e_{N}$. And under the null, the matrices $a\Sigma_{a,E}$ are defined with respect to the basis $f_1(\alpha\mu_E),..., f_{p}(\alpha\mu_E)$ of local frame fields.  We then have for each $a=1,...,g$
\begin{equation}n_a \tan_{j(\alpha\mu_E)}(j(a\bar{X}_{a,E})-j(\alpha\mu_E))^T\alpha\Sigma_{a,E}^{-1} \tan_{j(\alpha\mu_E)}(j(a\bar{X}_{a,E})-j(\alpha\mu_E))  \to_d   \chi_{d}^2, \end{equation} and since the $g$ populations are independent, we obtain
\begin{equation}\label{eq:limit-g-groups}
\sum_{a=1}^g n_a \tan_{j(\alpha\mu_E)}(j(a\bar{X}_{a,E})-j(\alpha\mu_E))^T\Sigma_{a,E}^{-1}~\tan_{j(\alpha\mu_E)}(j(a\bar{X}_{a,E})-j(\alpha\mu_E))  \to_d   \chi_{gd}^2.
\end{equation}
Under the null hypothesis of \eqref{eq:h0-MANOVA}, $a\bar{X}_{b,E}$ is a consistent estimator of $\alpha\mu_E$, then the embedding of the pooled extrinsic sample antimean
\begin{align}
j(a\bar{X}_{E})= P_{F,j} \left( \frac{1}{n} \sum_{a=1}^{g} n_a j(a\bar{X}_{a,E})\right)  \rightarrow_{p} j(\mu_{E})
 \end{align}

And since $\forall a = 1, \dots g, aS_{a,E}$ consistently estimate $\alpha\Sigma_{a,E}$ and $\tan_{j(a\bar{X}_{E})}$ is a consistent estimator of $\tan_{j(\alpha\mu_E)}$, we obtain \eqref{eq:manova-large} and \eqref{eq:manova-large2}.

\end{proof}
\begin{corollary}\label{cor:conf-reg-general-null}
Under the null hypothesis in \eqref{eq:h0-MANOVA}, confidence  regions for $\alpha\mu_E$ of asymptotic level $ 1 - c $ are given by $C_{n, c}^{(g)}$ and $D_{n, c}^{(g)}$ as follows
\begin{itemize}
\item $C_{n, c}^{(g)}= j^{-1}(U_{n , c})$ where \\$U_{n, c}= \{  j (\nu) \in j(\mathcal{M}): \sum_{a=1}^g n_a  \left\| aS_{a,E}^{-1/2}~\tan_{ j (\nu)}(j(a\overline{X}_{a,E})- j (\nu) )\right\|^2 \leq \chi^{2}_{gd, 1-c} \}$
\item $D_{n, c}^{(g)}= j^{-1}(V_{n , c})$ where \\$V_{n, c}= \{  j(\nu) \in j(\mathcal{M}):  \sum_{a=1}^g n_a \left\| aS_{a,E}^{-1/2}~\tan_{j(\bar{X}_{E})}(j(a\overline{X}_{a,E})-j(\nu) )\right\|^2 \leq \chi^{2}_{gd, 1-c} \}$
\end{itemize}
where $a\overline{X}_{a,E}$ is the pooled extrinsic sample antimean.
\end{corollary}
For $a=1,...,g,$ let $~\{ X_{a,i} \}_{i=1,\dots, n_a}$ be i.i.d.r.o's from the $\alpha j$-nonfocal distributions $\mathcal{Q}_a.$ Let $\{X_{a,r}^{*} \}_{r=1,...,n_a}$ be random resamples with repetition  from the empirical $\hat{Q}_{n_a}$ conditionally given $\{ X_{a,i} \}_{i=1,\dots, n_a}.$ The confidence regions $C_{n, c}^{(g)}$ and $D_{n, c}^{(g)}$ described in Corollary \ref{cor:conf-reg-general-null}  have corresponding bootstrap analogues as given below.

\begin{corollary}\label{c2}
The $(1-c) 100 \%$ bootstrap confidence regions for $\alpha\mu_{E}$ with $d=g p$ are given by ${C^*}_{n, c}^{(g)}= j^{-1}(U^{*}_{n , c})$ and
\begin{equation}\label{cstar_g}
U^{*}_{n , c}= \{  j(\nu) \in j(\mathcal{M}):   \sum_{a=1}^g n_a \left\| aS_{a,E}^{-1/2}~\tan_{j(\nu)}(j(a\overline{X}_{a,E})-j(\nu))\right\|^2 \leq {c^*}^{(g)}_{1-c} \},
\end{equation}
where ${c^*}^{(g)}_{1-c}$ is the upper $100(1-c) \%$ point of the values
\begin{equation}
 \sum_{a=1}^g n_a \left\|  {aS_{a,E}^{*}}^{-1/2}~\tan_{j(a\bar{X}_{E})}(j(a\overline{X^*}_{a,E})-j(a\bar{X}_{E}))\right\|^2
\end{equation} among all bootstrap resamples, and ${D^*}_{n, c}^{(g)}= j^{-1}({V^*}_{n , c})$, with
\begin{equation}\label{dstar_g}
{V^*}_{n, c}= \{ j(\nu ) \in j(\mathcal{M}):  \sum_{a=1}^g n_a \left\| aS_{a,E}^{-1/2}~\tan_{j(\bar{X}_{E})}(j(a\overline{X}_{a,E})-j(\nu) )\right\|^2  \leq {d^*}^{(g)}_{1-c} \}
\end{equation}
where ${d^*}^{(g)}_{1-c} $ is the upper $100(1-c)\%$ point of the values
\begin{equation}
\sum_{a=1}^g n_a \left\|  {{aS_{a,E}^{*}}}^{-1/2}~\tan_{j({\alpha\bar{X}^*}_{E})}(j(a\overline{X^*}_{a,E})-j(\alpha\bar{X}_{E}))\right\|^2,
\end{equation}
and ${a\bar{X}^*}_{E}$ is the extrinsic pooled bootstrap sample antimean given by
\begin{equation}\label{E_Polled_sample_mean}
j({a\bar{X}^*}_{E})= P_j \left( \frac{n_1}{n} j(a\bar{X}^{*}_{1, E}) + \cdots + \frac{n_g}{n}j(a\bar{X}^{*}_{g, E}) \right).
\end{equation}
Both confidence regions given by \eqref{dstar_g} and \eqref{cstar_g} have coverage error $O_{p}(n^{-2}).$
\end{corollary}

\noindent Note that
\begin{eqnarray}
aS_{a,E}^{*}=\left[ \left[\sum_{a=1}^{d} d_{a\overline{j^{(p)}(X^{*})}}P_{j}(e_b) \cdot e_{i}(j(\alpha\bar{X}^{*}_{E}))  ~e_{i}(j(\alpha\bar{X}^{*}_{E}))\right]_{i=1,...,p} \right] \cdot ~S^{*}_{n_a}\nonumber\\
\left[ \left[\sum_{a=1}^{m} d_{a\overline{j^{(p)}(X^{*})}}P_{j}(e_b) \cdot e_{i}(j(\alpha\bar{X}^{*}_{E})) e_{i}(j(\alpha\bar{X}^{*}_{E})) \right]_{i=1,...,p} \right]^T \nonumber
\end{eqnarray}
where $\displaystyle{S^{*}_{n_a}= (n_a )^{-1} \Sigma_{i=1}^{n_a}(j(X^{*}_{a,i})- j(\alpha\bar{X}^{*}_{E})) (j(X^{*}_{a,i})- j(\alpha\bar{X}^{*}_{E}))^T   }.$

\noindent In terms of nonparametric bootstrap approximations, for hypothesis testing, we will rely on the following result obtained
by substituting $$X^{(g)}= ({X_{1,a_1}})_{a_1=1,\dots,n_1},\cdots, ({X_{g,a_g}}_{a_g=1,\dots,n_g})$$
with resamples with repetition $$X^{*(g)}= ({X^{*}_{1,a_1}})_{a_1=1,\dots,n_1},\cdots, ({X^{*}_{g,a_g})})_{a_g=1,\dots,n_g}).$$

\begin{proposition}\label{p:asymptotic}
  For a $a=1,...,g$, let $\{ X_{a,i} \}_{i=1,\dots, n_a}$  i.i.d.r.o.'s from the $j$-nonfocal distributions $\mathcal{Q}_a.$  Let $\mu_a = E(j(X_{a,1}))$ and assume the extrinsic covariance matrice $a\Sigma_{a,E}$ of $X_{a,1}$ is nonsingular, $\forall a = 1, \dots, g.$
Then the distribution of
 $$T_{c}({X}^{(g)},{Q}^{(g)})= \sum_{a=1}^g n_a  \left\| a\Sigma_{a,E}^{-1/2}~\tan_{ j (\alpha\mu_E)}(j(a\overline{X}_{a,E})- j (\alpha\mu_E) )\right\|^2$$  can be approximated by the bootstrap distribution of\\
 ${T_{c}(X^{*(g)},\hat{Q}^{(g)})}=\sum_{a=1}^g n_a \left\|  aS_{a,E}^{-1/2}~\tan_{j(a\bar{X}_{E})}(j(a\overline{X}^{*}_{a,E})-j(a\bar{X}_{E}))\right\|^2.$
\item Similarly, the distribution of \\
$T_{d}({X}^{(g)},\hat{Q}^{(g)})=\sum_{a=1}^g n_a \left\| aS_{a,E}^{-1/2}~\tan_{j(a\bar{X}_{E})}(j(a\overline{X}_{a,E})-j(\alpha\mu_E) )\right\|^2 $  can be approximated by the bootstrap distribution function of\\
 ${T_{d}(X^{*(g)},\hat{Q}^{*(g)})}=\sum_{a=1}^g n_a \left\|  {aS_{a,E}^{*}}^{-1/2}~\tan_{j({a\bar{X}^*}_{E})}(j(a\overline{X^*}_{a,E})-j(a\bar{X}_{E}))\right\|^2$ with coverage error $O_P(n^{-2})$.
\end{proposition}

\subsection{VW Anti-MANOVA on $(\mathbb R P^{3})^q$}\label{ssc:manova-3Dproj}
\noindent In this subsection, we specialize the methods presented above in \ref{sc4} to anti-MANOVA on $P\Sigma_{3}^{k}$, the  projective shape space of 3D $k$-ads in $\mathbb RP^m$ for which $\pi=([u_1], \dots, [u_5])$ is a projective frame in $\mathbb RP^3,$ which is homeomorphic to the manifold $\left( \mathbb R P^3 \right)^{k-5}$ with $k-5=q$ (see Patrangenaru et. al (2010)\cite{PaLiSu:2010}). The embedding on this space is the VW embedding given in \eqref{eq:VW-emb}.

Additionally, from Proposition \ref{p:aVWprsh}, the corresponding farthest projection
\begin{eqnarray}\label{VW_Proj}
P_{j_q,F}: \left( S_{+}(4, \mathbb R )\right)^q \backslash  \mathcal F_q \to  j_k\left( \mathbb R P^3 \right)^q \notag\\
P_{j_q,F}(A_1, \dots, A_q)= \left(  j([m_1]), \dots, j[m_q])\right)
\end{eqnarray}
where $\forall a = 1,\dots, q, m_a$ is an eigenvectors of norm one of $A_a$, corresponding to its lowest eigenvalues, which is simple. That is same as saying that if $Y$ is a random object from a distribution $Q$ on $(\mathbb{R}P^3)^{q},$ where $Y = (Y^1, \dots, Y^q),$ and $Y^s= [X^s] \in \mathbb RP^3, s = 1, \dots, q,$ then T]the VW antimean of $Y$ is given by
\begin{equation}\label{eq:VW-mean-q}
\alpha\mu_{j_q}=([\gamma_{1}(1)], \cdots,[\gamma_q(1)]),
\end{equation}
 where, for $ s = \overline{1,q},~ \lambda_s(r)$ and $\gamma_s(r), r=1, \dots , 4$ are the eigenvalues in increasing order and the corresponding eigenvectors of $E\left[ X^s (X^s)^T \right].$
Given i.i.d.r.o.'s $Y_1, \dots, Y_n$ from a distribution $Q$ on $(\mathbb{R}P^3)^{q},$ with $Y_i = (Y^1_i, \dots, Y^q_i),$ and $Y^s_i= [X^s_i], {X^s_i}^TX^s_i=1,$ their sample VW-antimean is given in \eqref{eq:VW-sample-mean}, for $m=3$. The VW-anticovariance matrix ( anticovariance matrix associated with the VW embedding $j_q$) derived from  \eqref{eq:anti-cov} has the entries

\begin{eqnarray}\label{eq:anticov-VW-q}
aS_{j_q,(s,a),(t,b)}= n^{-1} (d_s(1)-d_s(a))^{-1} (d_t(1)- d_t(b))^{-1} \times \nonumber \\
\sum_{i=1}^{n} (g_s(a) \cdot X^s_i)(g_t(b) \cdot X^T_i)(g_s(1) \cdot X^s_i) (g_t(1) \cdot X^T_i),
\end{eqnarray}
for the pair of indices $(s,a),(t,b), s,t=1,\dots, q$ and $a, b = 2,3,4,$ listed in their lexicographic order, where for $a = 1,\dots, 4,$ $d_{s}(a), g_{s}(a)$ are the respectively eigenvalues in increasing order and  corresponding unit eigenvectors of
\begin{equation}\label{eq:j-s} J_{s}=n^{-1}\sum_{r=1}^{n}X_{r}^{s}(X_{r}^{s})^{T}
\end{equation}

Assume the VW anticovariance matrix $\alpha\Sigma_{j_q}$ is positive definite, thus given a large sample, with high probability the sample VW anticovariance matrix $aS_{j_q}$ has an inverse. Then, asymptotic distribution of the corresponding Hotelling $T^2$ type r.v.
\begin{equation}\label{eq:t2-anti}
T(Y,\alpha\mu_{j_q})= n \| {aS_{j_q}}^{-1/2} \tan_{j(a\overline{Y}_{j_q})} \left( j_q(a\overline{Y}_{j_q}) - j(\alpha\mu_{j_q})\right) \|^2
\end{equation}
is a $\chi^2_{3q},$ and its expression is
\begin{eqnarray}
T(Y,([\gamma_{1}(1)], \cdots,[\gamma_q(1)]))= n~\left( \gamma_{1}(1)^T D_1 \dots \gamma_{q}(1)^T D_q \right)~~{aS_{j_q}}^{-1}\cdot \nonumber \\ \cdot\left( \gamma_{1}(1)^T D_1 \dots \gamma_{q}(1)^T D_q \right)^T
\end{eqnarray}
where $aS_{j_q}$ and $D_s = (g_s(2)~ g_s(3) ~g_{s}(4)) \in \mathcal M(4,3,\mathbb R)$ are given as in \eqref{eq:anticov-VW-q}. We are in the position of giving the explicit expression of the test statistics that are addressing the VW anti-MANOVA hypothesis testing problem:

\begin{align}
\label{eq:hypot-vw-antimanova}
H_0 &:\ \alpha\mu_{1,E}=\alpha\mu_{2,E}=...=\alpha\mu_{g,E}= \alpha\mu_E,\\
H_a &:\ at\ least\ one\ equality\ \alpha\mu_{a,E}=\alpha\mu_{b,E}, 1 \leq a < b\leq g \ does\ not\ hold. \notag
\end{align}

Let $Y^{(g)} = (Y_{a,1},\dots,Y_{a,n_a})_{a=1,\dots,g}$ be independent r.o.'s from the distributions $Q_a, a=1,\dots,g$ on $( \mathbb R P^3)^q.$ We aim at having an explicit representation of the expression of the second test statistic
\begin{equation}
T_{d}(Y^{(g)},\hat Q^{(g)})=\sum_{a=1}^{g} { n_a} \| aS_{a,j_q}^{-1/2} \tan_{j_{q}}(a\bar{Y}_{j_q}) \left(j_{q}(a\bar{Y}_{j_q}) - j_{q}(\alpha\mu_{j_q}) \right)\|^2,
\end{equation}
from Proposition \ref{p:asymptotic}, where $\alpha\mu_{j_q},a\bar{Y}_{j_q}$ are respectively the pooled VW antimean and pooled sample VW antimean for the given data. Note that $\alpha\mu_{a,j_q}= ([\gamma_{1}^{a}(1)], \dots, [\gamma_{q}^{a}(1)])$ is the VW antimean for the sample from distribution $Q_a$ (of $Y_{a,1}, \dots, Y_{a,n_a}$) and $(\eta_{s}^{a}(r), \nu_{s}^{a}(r)),$ { $r=1,\dots,4$,} are eigenvalues and corresponding unit eigenvectors of $E(X_{a,1}^{s}(X_{a,1}^{s})^T]$.  The corresponding VW sample antimean is given by $a\overline{Y}_{a,j_q}=([g_{1}^{a}(1), \dots, [g_{q}^{a}(1)]),$ where for each $s=1, \dots, q$ and $r=1,\dots, 4$, $(d_{s}^{a}(r), g_{s}^{a}(r))$ are eigenvalues in increasing order and corresponding unit eigenvectors of $J_{s}^{a}= \frac{1}{n_a} \sum_{i=1}^{n_a}X^{s}_{a,i} (X^{s}_{a,i})^T.$ Also $\alpha\mu_{j_q}$ is the VW pooled antimean given by
\begin{equation}
j_{q}(\alpha\mu_{j_q})= P_{j_q,F}\left( \sum_{a=1}^{g} \lambda_a j_{q}(\alpha\mu_{a,j_q}) \right)\\
\alpha\mu_{j_q}=([\gamma_{1}^{(p)}(1)], \dots, [\gamma_q^{(p)}(1)]),
\end{equation}
where for $s=1,\dots, q, \gamma_{1}^{(p)}(1)$ is the eigenvector corresponding to the smallest eigenvalue of the $s-th$ axial component of the pooled matrix
with weights $\lambda_a, a=1,\dots, g, \sum_a \lambda_a =1, \lambda_a >0$ given by
$$\sum_{a=1}^g {\lambda_a}E(X_{a,1}X_{a,1}^T).$$

The pooled VW-sample antimean $a\overline{Y}^{(p)}_{j_q}$ is given by
\begin{eqnarray}
j_{q} \left(a\overline{Y}_{j_q} \right)= P_{j_q,F}\left( \sum_{a=1}^{g} \frac{n_a}{n} j_{q}(a\overline{Y}_{a,j_q})  \right)\\
a\overline{Y}^{(p)}_{j_q}= ([{\bf g}_{1}^{(p)}(1)], \dots, [{\bf g}_{q}^{(p)}(1)]).
\end{eqnarray}
Here for $s=1,\dots,q$,  ${\bf d}_{s}^{(p)}(r)$ and ${\bf g}_{s}^{(p)}(r) \in \mathbb R^4,~r=1,2,3,4,$ are eigenvalues in increasing order and corresponding unit eigenvectors of the matrix { $J^{(p)}=\sum_{a=1}^{g} \frac{n_a}{n} j_{k}(\overline{Y}_{a,E}) .$}\\

The following matrices
\begin{eqnarray}\label{eq:basis-tang-sample-pooled}
{\bf D}_s=({\bf g}_{s}^{(p)}(2)~ {\bf g}_{s}^{(p)}(3) ~{\bf g}_{s}^{(p)}(4)) \in \mathcal M (4,3:\mathbb R)
\end{eqnarray}
are giving a basis in the tangent space of the pooled sample VW antimean.

\begin{theorem}\label{t:chi-square-vw}
Assume $\{Y_{a, r_a}\}_{r_a=1, \dots,n_a},~a=1, \dots,g$ are i.i.d.r.o.'s from  the $j_q$-nonfocal probability measures $Q_a$ on $(\mathbb RP^3)^q$ with the VW embedding of $j_q$ leading to nondegenerate $j_q$-extrinsic anticovariance matrices. Consider the statistic

\begin{eqnarray}T_{d}( Y^{(g)},a\overline{Y}_{j_q})= \sum_{a=1}^{g} n_a ~\left[(  \gamma_{1}^{(p)}(1)-g_{1}^{a}(1) )^T {\bf D}_1 \dots (\gamma_{q}^{(p)}(1)-g_{q}^{a}(1) )^T {\bf D}_q \right]~~ \nonumber \\
aS_{a,j_q}^{-1} \nonumber \\
\left[(  \gamma_{1}^{(p)}(1)-g_{1}^{a}(1) )^T {\bf D}_1 \dots (\gamma_{q}^{(p)}(1)-g_{q}^{a}(1) )^T {\bf D}_q \right]^T,
\end{eqnarray}
where
\begin{eqnarray}
 {aS_{a,j_q}}_{(s,c)(t,b)}= n_{a}^{-1} ({\bf d}_{s}^{(p)}(1) - {\bf d}_{s}^{(p)}(c))^{-1} ({\bf d}_{t}^{(p)}(1) - {\bf d}_{t}^{(p)}(b))^{-1} \nonumber\\
 \times \sum_{i} ({\bf g}^{(p)}_{s}(c)\cdot X_{a,i}^{s})({\bf g}^{(p)}_{t}(b) \cdot X_{a,i}^{t})({\bf g}^{(p)}_{s}(1) \cdot X_{a,i}^{s})({\bf g}^{(p)}_{t}(1) \cdot X_{a,i}^{t}) \nonumber
 \end{eqnarray}
  and $s,t=1, \dots,q$ and $c,b= 2, 3, 4.$
If $\frac{n_a}{n} \to \lambda_a >0,$ as $n \to \infty,$ then $T_{d}( Y^{(g)},a\overline{Y}_{j_q})$ converges weakly to a $\chi^{2}_{3q}$ distributed r.v.
\end{theorem}

\begin{proof}
The asymptotic behaviors of the sample VW-antimeans follow from Theorem \ref{t:chi-antimean}, when applied to the VW embedding $j_q$ of
$(\mathbb RP^3)^q$ (a.k.a. projective shape space of $q+5$-ads in general position in $\mathbb RP^3$ ) given in \eqref{eq:VW-emb}.
Indeed from \eqref{eq:manova-large}, we split the difference $j_q(a\bar{X}_{a,j_q})-j(a\bar{X}_{j_q})$ in the tangent space
$T_{j_q(a\bar{X}_{a,j_q})}(j_q(\mathbb RP^3)^q),$ w.r.t. the orthogonal basis described in \eqref{eq:basis-tang-sample-pooled}.
\end{proof}

\begin{corollary}\label{c:boot-conf-reg-vw-antimean}
A $(1-c) 100 \%$ nonparametric bootstrap confidence region for $\alpha\mu_{j_q}$ is given by
\begin{equation}\label{eq:boot-conf-reg-vw-antimean}{D^*}_{n, c}^{(g)}= j^{-1}({V^*}_{n , c}),
\end{equation}
where
${V^*}_{n, c}= \{ j_k(\nu), T_{d}( Y^{(g)},a\overline{Y}_{j_q}, \nu) \leq  {d^*}^{(g)}_{1-c} \}$ and
$T_{d}( Y^{(g)},a\overline{Y}_{j_q}, \nu )={n_a} \sum_{a=1}^{g} \left\| aS_{a,j_q}^{-1/2} \tan_{j_{q}(a\overline{Y}_{j_q})} (j_{q}(a\overline{Y}_{j_q}) - j_{q}(\nu))\right\|^2$
with ${d^*}^{(g)}_{1-c} $ being the upper $100(1-c)\%$ point of the values of
\begin{equation}\label{eq:manova-boot}
{T_{d}({Y^*}^{(g)},a\overline{Y}^*_{j_q}, a\overline{Y}_{j_q})=\sum_{a=1}^{g} n_a~ \left\| {aS_{a,j_q}^{*}}^{-1/2}~\tan_{j_{q}(a\overline{Y}^*_{j_q})}(j_{q}(a\overline{Y}^{*}_{a,j_q}) - j_{q}(a\overline{Y}_{j_q}))\right\|^2},
\end{equation}
among the bootstrap resamples, where
\begin{eqnarray}
 {aS_{a,j_q}^{*}}_{(s,c)(t,b)}= n_{a}^{-1} ({\bf d}_{s}^{*(p)}(1) - {\bf d}_{s}^{*(p)}(c))^{-1} ({\bf d}_{t}^{*(p)}(1) - {\bf d}_{t}^{*(p)}(b))^{-1}\times \nonumber \\
 \sum_{i} ({\bf g}^{*(p)}_{s}(c)\cdot X_{a,i}^{*s})({\bf g}^{*(p)}_{t}(b) \cdot X_{a,i}^{*t})({\bf g}^{*(p)}_{s}(1) \cdot X_{a,i}^{*s})({\bf g}^{*(p)}_{t}(1) \cdot X_{a,i}^{*t}) , b, c = 2,3,4.\nonumber
 \end{eqnarray}
The confidence regions given by \eqref{eq:boot-conf-reg-vw-antimean} has coverage error $O_{p}(n^{-2}).$
\end{corollary}

\section{Application to face data analysis}\label{sc:faces}
Digital images collected with a high resolution Panasonic-Lumix DMC-FZ200 camera, posted at ani.stat.fsu.edu/$\sim$vic/E-MANOVA, were used to test for a VW mean 3D projective shape difference between five faces (see Yao et al(2017)\cite{YaPaLe:2017}). The 3D surface reconstructions of those faces, with the seven labeled landmarks, and a projective frame are displayed in Figures in Yao et al(2017)\cite{YaPaLe:2017}.

 Here we compare the projective shapes of these faces by first conducting a VW anti-MANOVA analysis on $P \Sigma_{3}^{7}=(\mathbb RP^3)^{2},$
 testing the hypotheses \eqref{eq:hypot-vw-antimanova}, based on the sample sizes on hand: $n_1=n_2=n_4=n_5=6$ and $n_3=7,$ the null hypothesis being rejected if
$$T_{d}( {Y}^{(5)},a\overline{Y}_{j_2})=\sum_{a=1}^{5} n_a~ \left \| S_{\alpha\bar Y_{a},j_2}^{-1/2}~\tan_{j_{2}(a\overline{Y}_{j_2})} (j_{2}(a\overline{Y}_{a,j_2}) - j_{2}(a\overline{Y}_{j_2})) \right\|^2$$ is greater than ${d^*}^{(5)}_{1-\alpha},$
where ${d^*}^{(5}_{1-\alpha}$ is the $(1-\alpha)100\%$ cutoff of the corresponding bootstrap distribution in equation \eqref{eq:manova-boot}. With $5,000$ bootstrap resamples, we obatain $T_{d}( {y}^{(5)},a\overline{y}_{j_2})= 26,848.81$, and their corresponding empirical $p$-value  $0.0088$. Thus concluding that there exists a statistically significant VW-antimean 3D-projective shape face difference between at least two of the individuals in our data set.
\noindent We then ran pairwise tests for antimean projective shape changes from subsection \ref{ssc:2vw-anti}, and obtained the following table

\begin{table}[H]
\footnotesize
\centering
\begin{tabular}{llllllllllllllllll}
  \hline
 & (1,2) & (1,3) & (1,4)  & (1,5) & (2,3) & (2,4) & (2,5) & (3,4) & (3,5) & (4,5) \\
  \hline
Test result & Reject & Reject & Reject &   No &   Reject &   Reject  & No &   No &  No &  No \vspace{.1cm}\\
   \hline
\end{tabular}\caption{Results of pairwise VW mean change}\label{facetest}
\end{table}

\section*{Discussion}

In this paper in the case of a r.o. on a compact manifolds, we study a recently introduced location parameter, the Fr\'echet antimean (set), with en emphasis on the extrinsic antimean, leading to new statistics, such as the sample extrinsic antimean, and the sample extrinsic anticovariance matrix. Just as with the extrinsic mean, the extrinsic antimean captures important features of a distribution on a compact object space. More general location parameters, means of a given index were introduced in Patrangenaru et al.(2019)\cite{PaBuPaOs:2019}.  While our results extend to the general case of arbitrary Fr\'echet antimean, for the purpose of data analysis, extrinsic antimeans are likely to be faster to compute (see Bhattacharya et al.(2012)\cite{BhElLiPaCr:2012}), and easier to characterize than their general Fr\'echet counterparts.
Therefore future research for extrinsic means of given indices, including stickiness of extrinsic antimeans and extrinsic antiregression (see
Deng et al (2018)\cite{DeBaPa:2018}), should parallel research on inference for extrinsic means (see Hotz et al(2013)\cite{Hotz:2013}, Bhattacharya et al(2014)\cite{BhElLiPaCr:2012}, Petersen and M\"{u}ller(2017)\cite{PeMu:2019}). Given that data comes these days in some form of electronic images, more emphasis should be put on collecting large samples of picture of 3D scenes, 3D image analysis, especially extracting and analyzing 3D projective shape and color info from digital camera images and medical imaging outputs.

\end{document}